\documentclass[12pt]{amsart}
\usepackage{amssymb,amsmath,amsthm,bm}
\usepackage{mathtools}
 \usepackage{mathptmx}
\usepackage[colorlinks=true,allcolors=blue]{hyperref}
\usepackage{todonotes}

\usepackage[T1]{fontenc}
\DeclareFontFamily{U}{mathx}{\hyphenchar\font45}
\DeclareFontShape{U}{mathx}{m}{n}{
      <5> <6> <7> <8> <9> <10>
      <10.95> <12> <14.4> <17.28> <20.74> <24.88>
      mathx10
      }{}
\DeclareSymbolFont{mathx}{U}{mathx}{m}{n}
\DeclareFontSubstitution{U}{mathx}{m}{n}
\DeclareMathAccent{\widecheck}{0}{mathx}{"71}
\DeclareMathAccent{\wideparen}{0}{mathx}{"75}
\DeclareSymbolFont{usualmathcal}{OMS}{cmsy}{m}{n}
\DeclareSymbolFontAlphabet{\mathcal}{usualmathcal} 
\usepackage{color}

\def\e{{\rm e}}
\def\cic{\bm}
\def\eps{\varepsilon}

\def\d{{\rm d}}
\def\dist{{\rm dist}}

\def\R {\mathbb{R}}

\def\B {{\mathsf B}}

\def\S{{\mathbf S}}

\def \l {\langle}
\def \r {\rangle}

\def \and{\qquad\text{and}\qquad}

\newcommand{\supp}{\mathrm{supp}\,}

\newcounter{thms}
\newtheorem{proposition}{Proposition}
\newtheorem{theorem}[thms]{Theorem}
\newtheorem*{theorem*}{Theorem}
\newtheorem*{proposition*}{Proposition}

\newtheorem{lemma}[proposition]{Lemma}
\theoremstyle{definition}

\newtheorem{remark}[proposition]{Remark}

\numberwithin{proposition}{section}
\numberwithin{equation}{section}

\title[Sparse domination of variational Carleson operators]{Positive sparse domination of variational Carleson operators}

 \author{Francesco  Di Plinio}
 \author{Yen Q.\ Do}
\address{Department of Mathematics,
The University of Virginia,\newline \indent  Charlottesville, VA 22904-4137, USA}
\email[F.\ Di Plinio]{francesco.diplinio@virginia.edu}
\email[Y. Do]{yen.do@virginia.edu}

\author{Gennady N.\ Uraltsev}
\address{ Mathematics Department,	Universit\"at Bonn, \newline \indent 
Endenicher Allee 60,	D - 53115 Bonn, Germany
  }
  \email[G.\ Uraltsev]{gennady.uraltsev@math.uni-bonn.de}

 \subjclass[2010]{Primary: 42B20. Secondary: 42B25}
 \keywords{Positive sparse operators,  variational Fourier series,  weighted norm inequalities}
\thanks{F. Di Plinio was partially
supported by the National Science Foundation under the grant
   NSF-DMS-1500449 and  NSF-DMS-1650810. Y.\ Do was partially supported by the National Science Foundation under the grant
 NSF-DMS-1521293. }

\begin{document}
\begin{abstract} Due to its nonlocal nature, the $r$-variation norm
  Carleson operator $C_r$ does not yield to the sparse domination
  techniques of Lerner \cite{Ler2015,Ler2013}, Di Plinio and Lerner
  \cite{DPLer14}, Lacey \cite{Lac2015}.  We overcome this difficulty
  and prove that the dual form to $C_r$ can be dominated by a positive
  sparse form involving $L^p$ averages. Our result strengthens the
  $L^p$-estimates by Oberlin et.\ al. \cite{OSTTW}. As a corollary, we
  obtain quantitative weighted norm inequalities improving on
  \cite{DoLac12F} by Do and Lacey. Our   proof 
  relies on the localized outer $L^p$-embeddings of Di Plinio and Ou
  \cite{DPOu2} and Uraltsev \cite{Ur15}. 
 \end{abstract}
\maketitle
\section{Introduction and main results}

The technique of controlling Calder\'on-Zygmund singular integrals,
which are a-priori \emph{non-local}, by \emph{localized} positive sparse operators
has recently emerged as a leading trend in Euclidean Harmonic
Analysis. We briefly review the advancements which are most relevant
for the present article and postpone further references to the body of
the introduction.  The original domination in norm result of
\cite{Ler2013} for Calder\'on-Zygmund operators has since been
upgraded to a \emph{pointwise} positive sparse domination by Conde and
Rey \cite{CondeRey2015} and Lerner and Nazarov \cite{LerNaz2015}, and
later by Lacey \cite{Lac2015} by means of an inspiring stopping time
argument forgoing local mean oscillation. Lacey's approach was further
clarified in \cite{Ler2015}, resulting in the following principle: if $T$ is a sub-linear operator
of weak-type $(p,p)$ and in addition the maximal operator
\begin{equation}
\label{nonloc}
f\mapsto \sup_{  Q \subset \R \textrm{ interval}} \left\|T(f\cic{1}_{\R\setminus 3Q}) \right\|_{L^\infty(Q)} \cic{1}_Q 
\end{equation}
embodying the \emph{non-locality} of $T$, is of weak-type $(s,s)$, for some $1\leq p\leq s<\infty$, then $T$ is pointwise dominated by a positive sparse operator involving $L^s$ averages of $f$.

 The   principle \eqref{nonloc}  extends to certain modulated singular integrals.  Of interest for us is  the maximal partial Fourier transform   
\[
  Cf(x)=\sup_{N} \left|\int_{\infty}^{N} \widehat f(\xi)
    \,\e^{ix\xi}\d\xi\right|\] also known as Carleson's operator on the
real line. The crux of the matter is that \eqref{nonloc} follows for
$T=C$ from its representation as a maximally modulated Hilbert
transform, a fact already exploited in the classical weighted norm
inequalities for $C$ by Hunt and Young \cite{HY}, and in the more
recent work \cite{GMS}.  Together with sharp forms of the
Carleson-Hunt theorem near the endpoint $p=1$ \cite{DP2} this allows,
as observed by the first author and Lerner in \cite{DPLer14}, the
domination of $C$ by sparse operators and thus leads to sharp weighted
norm inequalities for $C$.

  In this article we consider the  $r$-variation norm Carleson operator, which is defined for Schwartz functions on the real line as  
\[
C_r f(x) =  \sup_{N\in \mathbb N} \sup_{ \xi_0< \cdots< \xi_{N}} \left(\sum_{j=1}^N \left|\int_{\xi_{j-1}}^{\xi_{j}} \widehat f(\xi)   \,\e^{ix\xi}\d\xi\right|^r\right)^{1/r}.
\]
The importance of $C_r$ is revealed by the transference principle,
presented in \cite[Appendix B]{OSTTW}, which shows how $r$-variational
convergence of the Fourier series of $f\in L^p(\mathbb{T};w)$ for a
weight $w$ on the torus $\mathbb T$ follows from
$L^p(\R; w)$-estimates for the sub-linear operator $C_r$.  Values of
interest for $r$ are $2<r<\infty$. Indeed the main result of
\cite{OSTTW} is that in this range, $C_r$ maps into $L^p$ whenever
$p>r'$, while no $L^p$-estimates hold for variation exponents
$r\leq 2$.  Unlike   the Carleson operator, its variation norm
counterpart $C_r$ does not have an explicit kernel form and thus fails to yield
to Hunt-Young type techniques.  The same essential difficulty is
encountered in the search for $L^q$-bounds for the nonlocal maximal
function \eqref{nonloc} when $T=C_r$. Therefore, the   
approach via \eqref{nonloc} does not seem to be applicable to $C_r$. In the series
\cite{DoLac12F,DoLac12W}, the second author and Lacey circumvented
this issue through a direct proof of $A_p$-weighted inequalities for
$C_r$ and its Walsh analogue, based on weighted phase plane analysis.

The main result of the present article is that a sparse domination
principle for $C_r$ holds in spite of the difficulties described
above. More precisely, we sharply dominate the dual form to the
$r$-variational Carleson operator $C_r$ by a single positive sparse
form involving $L^p$ averages, leading to an effortless strengthening
of the weighted theory of \cite{DoLac12F}. Our argument abandons  
\eqref{nonloc}   in favor of a stopping time construction,
relying on the localized Carleson embeddings for suitably modified
wave packet transforms of \cite{DPOu2} by the first author and Yumeng
Ou, and \cite{Ur15} by the third author. In particular, our technique
requires no \emph{a-priori} weak-type information on the operator $T$. A similar approach was   employed by
Culiuc, Ou and the first author in \cite{CuDPOu} in the proof of a
sparse domination principle for the family of modulation invariant
multi-linear multipliers whose paradigm is the bilinear Hilbert
transforms. Interestingly, unlike \cite{CuDPOu}, our construction of the sparse collection in Section \ref{secpf1} seems to be the first in literature which does not make any use of dyadic grids.

 We believe that intrinsic sparse domination can prove
useful in the study of other classes of multi-linear operators lying
way beyond the scope of Calder\'on-Zygmund theory, such as the
iterated Fourier integrals of \cite{DoMuscTh} and the sub-dyadic
multipliers of \cite{BenBeltr15}.

 To formulate our main theorem, we recall the notation 
\[
\l f \r_{I,p}:=\Big(\frac{1}{|I|}\int |f|^p\, \d x\Big)^{\frac1p}, \qquad 1\leq p<\infty
\]
where $I \subset \R$ is any interval, and the notion of a \emph{sparse collection}
of intervals. We say that the countable collection of intervals
$ I\in \mathcal S$ is $\eta$-sparse for some $0<\eta\leq 1$ if there
exist a choice of measurable sets $\{E_I\,\colon I \in \mathcal \S\}$ such that
\[
E_I \subset I,\qquad |E_I| \geq \eta |I|, \qquad E_I \cap E_J  = \varnothing
\quad \forall  I,J\in \mathcal S,\; I \neq J.
\]
\begin{theorem} \label{thmmain}
Let $2<r<\infty$ and $p>r'$. Given $f,g\in \mathcal
C_0^\infty(\R)$   there exists a   sparse collection $\mathcal S= \mathcal S(f,g,p)$ and an absolute constant $K=K(p)$ such that
\begin{equation}
\label{sparsedom}
|\l C_r f, g\r| \leq K(p) \sum_{I \in \mathcal S} |I| \l f\r_{I,p}  \l g\r_{I,1}. 
\end{equation}
\end{theorem}
A corollary of Theorem \ref{thmmain} is that $C_r$ extends to a bounded
sub-linear operator on $L^q(\R)$ whenever $q>r'$.  As a matter of fact,
let us fix $q\in (r',\infty]$, and choose $p \in(r',q)$. Denoting by
\[
  \mathrm M_p f(x) =\sup_{I\ni x} \l  f \r_{I,p}
\]
the $p$-th Hardy-Littlewood maximal function, the estimate of Theorem \ref{thmmain} and
the  fact that $\mathcal S$ is sparse yields
\[
  |\l C_r f, g\r| \lesssim \sum_{I \in \mathcal S} |E_I| \l f\r_{I,p} \l
  g\r_{I,1} \leq \l \mathrm M_{p} f, \mathrm{M}_1g \r \lesssim \| \mathrm M_{p}
  f\| _q \| \mathrm M_1g\|_{q'} \lesssim \|f\|_q\|g\|_{q'}.
\]
Bounds on $L^{q}$ for $C_{r}$ were
first proved in \cite{OSTTW}, where it is also shown that the
restriction $q>r'$ is necessary, whence no
sparse domination of the type occurring in Theorem \ref{thmmain} will
hold for $p<r'$. We can thus claim that Theorem \ref{thmmain} is sharp, short of the endpoint $p=r'$. In fact, sparse domination as in
\eqref{sparsedom} also entails $C_r:L^{p}(\R) \to L^{p,\infty}(\R)$. Such an estimate is currently unknown for $p=r'$.

However, Theorem \ref{thmmain} yields much more precise information
than mere $L^q$-boundedness. In particular, we obtain precisely
quantified weighted norm inequalities for $C_r$. Recall the definition
of the $A_t$ constant of a locally integrable nonnegative function $w$
as
\[
[w]_{A_t} := \begin{cases} \displaystyle
\sup_{I\subset \R} \, \l w \r_{I,1} \big\l w^{\frac{1}{1-t}}\big\r_{I,1}^{t-1} &
1<t<\infty \\ \inf\big\{A: \mathrm{M}w(x)\leq A w(x) \;\textrm{for a.e. }x\big\} &t=1\end{cases} 
\]
\begin{theorem} \label{thmweight} Let $2<r<\infty$ and $q>r'$ be fixed. Then
\begin{itemize}
\item[(i)] there exists  $K:[1,\frac{q}{r'})\to(0,\infty)$ nondecreasing such that
\[
\|C_{r}\|_{  L^q(\mathbb R;w) \to L^q(\mathbb R;w)} \leq K(t) [w]_{A_t}^{\max \left\{ 1, \frac{t}{q(t-1)} \right\}};
\]
\item[(ii)] there exists a positive increasing function $\mathcal Q$ such that for $t=\frac{q}{r'}$
\begin{equation}
\label{thmweight1}
\|C_{r}\|_{ L^q(\mathbb R;w) \to L^q(\mathbb R;w) } \leq \mathcal{Q}\left( [w]_{A_{t}} \right).
\end{equation}
\end{itemize}
\end{theorem}
We omit the standard deduction of Theorem \ref{thmweight} from Theorem \ref{thmmain}, which follows along lines analogous to the proofs of \cite[Theorem 3]{CuDPOu} and \cite[Theorem 17.1]{LerNaz2015}. Estimate (i) of Theorem \ref{thmweight} yields in particular that 
\[
w \in A_t \implies \|C_{r}\|_{ L^q(\mathbb R;w) \to L^q(\mathbb R;w) } <\infty \qquad \forall r> \max \left\{2,{\textstyle \frac{q}{q-t}}\right\}
\]
an improvement over \cite[Theorem 1.2]{DoLac12F}, where $L^q(\mathbb R;w)$ boundedness is only shown for variation exponents $r> \max \left\{2t,{\textstyle \frac{qt}{q-t}}\right\}$ when $w \in A_t$. Fixing $r$ instead, part (ii) of Theorem \ref{thmweight} is sharp in the sense that  $t=\frac{q}{r'}$ is the largest exponent such that an estimate of the type of \eqref{thmweight1} is allowed to hold. Indeed, if \eqref{thmweight1} were true for any $q=q_0\in (r',\infty)$ and  some $t=\frac{q_0}{s}$ with $s<r'$, a version of the Rubio de Francia extrapolation theorem (see for instance \cite[Theorem 3.9]{CruzMartellPerez})  would yield that $C_{r}$ maps $L^q$ into itself for all $q\in (s,\infty)$, contradicting the already mentioned counterexample from \cite{OSTTW}.

We turn to further comments on the proof and on the structure of the
paper. In the upcoming Section \ref{secwp} we reduce the bilinear form
estimate \eqref{sparsedom} to an analogous statement for a bilinear
form involving integrals over the upper-three space of symmetry
parameters for the Carleson operator of a wave packet transforms of
$f$ and a variational-truncated wave packet transform of $g$. The
natural framework for $L^p$-boundedness of such forms, the
$L^p$-theory of outer measures, has been developed by the second
author and Thiele in \cite{DoThiele15}. In Section \ref{secloc}, we
recall the basics of this theory as well as the localized Carleson
embeddings of \cite{DPOu2} and \cite{Ur15}. These will come to
fruition in Section \ref{secpf1}, where we give the proof of Theorem
\ref{thmmain}. A significant challenge in the course of the proof  is the treatment of the nonlocal (tail) components, which are handled  via novel \emph{ad-hoc} embedding theorems incorporating the fast decay of the wave packet coefficients away from the support of the input functions.

\subsection*{Acknowledgments}

This work was initiated and continued during G.\
Uraltsev's visit to the Brown University and University of Virginia Mathematics Departments, whose
hospitality is gratefully acknowledged. The authors would like to
thank Amalia Culiuc, Michael Lacey, Ben Krause and Yumeng Ou for
useful conversations about sparse domination principles.

\section{Reduction to wave packet transforms} \label{secwp}

In this section we reduce the inequality \eqref{sparsedom} to an
analogous statement involving wave packet transforms. Throughout this
section, the variation exponent $r\in (2,\infty)$ is fixed, and we
take $f,g\in \mathcal C^\infty_0(\R)$. 
 First of all we linearize the variation norm appearing in
$C_r$. Begin by observing that the map
\[
(x,\xi) \mapsto \int_{-\infty}^\xi \widehat f(\zeta) \, \e^{ix \zeta} \d  \zeta
\]
 is uniformly continuous. By
duality and standard considerations  
\[
C_r f(x) = \sup_{N}\;\sup_{\Xi \subset \R, \#\Xi\leq N} \;\sup_{\|\{a_j\}\|_{\ell^{r'}}\leq 1}\;  \sum_{j=1}^N a_j\; 
 \int_{\xi_{j-1}}^{\xi_j} \widehat f(\zeta) \, \e^{ix \zeta} \d  \zeta.
\]
Therefore,  \eqref{sparsedom} 
will be a consequence of the estimate  
\begin{equation}
\label{sparsedom1}
\begin{split}
\Lambda_{\vec \xi,\vec a }(f,g)&:=\int_\R g(x)\left(\sum_{j=1}^N a_j(x)   \int_{\xi_{j-1}(x)}^{\xi_j(x)} \widehat f(\zeta) \, \e^{ix \zeta}\, \d  \zeta\right)\d x  \leq
K(p) \sum_{I \in \mathcal S} |I| \l f\r_{I,p}  \l g\r_{I,1}
\end{split}
\end{equation}
with right hand side independent of $N\in \mathbb N$, $\Xi \subset \R, \#\Xi\leq N$, and of the   measurable $\Xi^{N+1}$-valued function $\vec \xi(x)=\{\xi_j(x)\}$ with $ \xi_0(x)<\cdots<\xi_N(x)$, and $\mathbb C^{N+1}$-valued $\vec a(x)=\{a_j(x)\}$ with $\|\vec a(x)\|_{\ell^{r'}}=1$.

The next step is to  uniformly dominate the form
$\Lambda_{\vec \xi,\vec a }(f,g)$ by an outer form involving wave
packet transforms of $f$ and $g$; in the terminology of
\cite{DoThiele15}, \emph{embedding} maps into the upper 3-space
\[(u,t,\eta) \in \mathbb X=\R\times \R_{+}\times \R.\] The parameters
$\xi, \vec a$ will enter the definition of the embedding map for
$g$. We introduce the wave packets
\[  \psi_{t,\eta}(x):=t^{-1}\e^{i\eta
    z}\,\psi\left(\frac{x}{t}\right), \qquad \eta\in \R,\; t\in (0,\infty) 
\]
where $\psi$ is a real valued, even Schwartz function with frequency
support of width $b$ containing the origin. The wave packet transform
of $f$ is thus defined, as in \cite{DoThiele15}, by
\begin{equation}
\label{wpt}
F(f)(u,t,\eta) = \big|f*\psi_{t,\eta} (u)\big|, \qquad (u,t,\eta) \in \mathbb X.
\end{equation}
For our fixed choice of $\vec \xi,\vec a$ we introduce the modified
wave packet transform of $g$ that is dual to \eqref{wpt} for the sake
of bounding the left hand side of \eqref{sparsedom1}. Following
\cite[Eq. (1.14)]{Ur15}, it is given by
\begin{equation}
\label{wptm}
     A(g)(u,t,\eta) := \sup_{\Psi} \left|
    \int_{\R}g(x)\sum_{j=1}^N
    a_{j}(x) \Psi_{t,\eta}^{\xi_{j}(x),\xi_{j+1}(x)}(x-u)\, \d x\right|, \qquad (u,t,\eta) \in \mathbb X,
\end{equation}
with supremum being taken over all choices of \emph{truncated wave
  packets} $\Psi_{t,\eta}^{\xi_{-},\,\xi_{+}}$, that for each
$t,\eta\in\R_{+}\times \R$ are functions in $\mathcal S(\R)$ parameterized by $\xi_{-},\xi_{+}\in\Xi$.  We summarize the
basic defining properties of the truncated wave packets in Remark \ref{twprem} below, and we
refer to \cite{Ur15} for a precise definition.

The duality of the embeddings \eqref{wpt} and \eqref{wptm} is a
consequence of  the following wave packet domination Lemma. We send to \cite{Ur15} for the proof.
\begin{lemma}[Wave packet domination]\label{lem:wave-packet-domination}
  Let $f,g$, $\Xi,\vec \xi,\vec a$ be as above and consider the
  bilinear form defined on the wave packets transforms, given by
  \begin{equation}
    \label{defB}
     \B_{\vec \xi,\vec a }(f,g):=\int_{\mathbb X}   F(f)(u,t,\eta)A(g)(u,t,\eta) \, \d u \d t \d \eta.
   \end{equation}
   Then
   \[
     \Lambda_{\vec \xi,\vec a }(f,g)\lesssim     \B_{\vec \xi,\vec a }\big(F(f),A(g)\big)
   \]
  with uniform implied constant. 
\end{lemma}

Using the above Lemma, we see that inequality  \eqref{sparsedom1}  and 
thus  Theorem \ref{thmmain} will follow from the bounds of the next
proposition.

\begin{proposition} \label{propmain} Let $p>r'$ be fixed.  For all
  $f,g\in L^\infty(\R)$ and compactly supported  there exists a sparse collection
  $\mathcal S= \mathcal S(f,g,p)$ and an absolute constant $K=K(p)$
  such that
\begin{equation}
\label{outer-sparsedom}
\sup_{N}\;\sup_{ \#\Xi\leq N}\; \sup_{\vec \xi, \vec a}\;
\B_{\vec \xi,\vec a }(f,g) \leq K(p) \sum_{I \in \mathcal S} |I| \l f\r_{I,p}  \l g\r_{I,1}
\end{equation}
where $\vec \xi, \vec a $ range over $\Xi^{N+1},\mathbb C^{N+1}$-valued functions as above.
\end{proposition}

 We now make a brief digression to
justify definitions \eqref{wpt} and \eqref{wptm} of the
wave packet transforms and the result of Lemma
\ref{lem:wave-packet-domination}. Consider the term
\[  
  \int_{\xi_{j-1}(x)}^{\xi_{j}(x)}\widehat{f}(\eta)e^{ix\eta}\d
  \eta \]
appearing in \eqref{sparsedom1} and let us think for a moment of
$\xi_{j-1}(x)=\xi_{-}$ and $\xi_{j}(x)=\xi_{+}$ as frozen. Then the following
representation holds for the multiplier
$1_{(\xi_{-},\xi_{+})}(\zeta)$:
\begin{align}\label{eq:multiplier-representation}
  1_{(\xi_{-},\xi_{+})}(\zeta) = \int_{\R_{+}\times\R}
  \widehat{\Psi}_{t,\eta}^{\xi_{-},\xi_{+}}(\zeta)\, \d t \d \eta 
\end{align}
where $\Psi_{t,\eta}^{\xi_{-},\xi_{+}}$ are truncated wave packets.
Choosing a $\phi\in \mathcal S(\R)$ such that
$\widehat{\phi}_{t,\eta}(\zeta)=1$ whenever
$\widehat{\Psi}_{t,\eta}^{\xi_{-},\xi_{+}}(\zeta) \neq 0$ for any
$\xi_{-}<\xi_{+}\in\R$, we obtain the pointwise identity
\[
  \int_{\xi_{j-1}(x)}^{\xi_{j}(x)}\widehat{f}(\zeta)e^{ix\zeta}\d \zeta  =\int_{\mathbb X}
  f*\phi_{t,\eta} (u) \Psi_{t,\eta}^{\xi_{j-1}(x),\xi_{j}(x)}(x-u) \d u \d t \d \eta.
\]
The results of Lemma \ref{lem:wave-packet-domination} follow by Fubini
and the triangle inequality.

We briefly illustrate identity \eqref{eq:multiplier-representation},
for a more careful discussion we refer to \cite[Sec. 3]{Ur15}. Start
by choosing $\psi\in \mathcal S(\R)$ with $\widehat{\psi}$ non-negative and
supported on a ball of radius $b/2$, and let
$\chi\in \mathcal{C}^{\infty}_{0}(\R)$ be a non-negative bump function supported
on $[d-\epsilon,d+\epsilon]$ with $d>b$ and $\epsilon\ll b$.   Suppose 
formally that $\xi_{+}=+\infty$ so that, up to a
suitable normalization of $\chi$, a Littlewood-Paley type
decomposition centered at $\xi_{-}$ of the multiplier
$1_{(\xi_{-},+\infty)}$ gives
\begin{align*}
  1_{(\xi_{-},+\infty)}(\zeta) = \int_{\R_{+} \times \R} \widehat
  \psi(t(\zeta-\eta)) \chi(t(\eta-\xi_{-})) \d t \d \eta.
\end{align*}
A similar  expression holds if $\xi_{-}=-\infty$ and $\xi_{+}\in
\R$. We choose truncated wave packets so that 
\begin{align*}
  &\Psi_{t,\eta}^{\xi_{-},\xi_{+}}(x):=
    \psi_{t,\eta}(x)\chi(t(\eta-\xi_{-}))&&t(\eta-\xi_{-})\ll
                                       t(\xi_{+}-\eta)\\
  &\Psi_{t,\eta}^{\xi_{-},\xi_{+}}(x):=
    \psi_{t,\eta}(x)\chi(t(\xi_{+}-\eta))&&t(\eta-\xi_{-})\gg
                                       t(\xi_{+}-\eta)\\
  &\Psi_{t,\eta}^{\xi_{-},\xi_{+}}(x):=0 &&\eta\notin(\xi_{-},\xi_{+}).
\end{align*}
Finally if $t(\eta-\xi_{-})\approx t (\xi_{+}-\eta)$ then
$\Psi_{t,\eta}^{\xi_{-},\xi_{+}}$ is chosen to appropriately model the
transition between the above regimes and justifies identity
\eqref{eq:multiplier-representation}.  
\begin{remark}
\label{twprem}In general we call a function
$\Psi_{t,\eta}^{\xi_{-},\xi_{+}}\in \mathcal S(\R)$ parameterized by
$\xi_{-}<\xi_{+}\in\R$ a   truncated   wave-packet adapted to
$t,\eta\in\R_{+}\times\R$ if 
\begin{align*}
  &e^{-i\eta t z } t \Psi_{t,\eta}^{\xi_{-},\xi_{+}}(tx),
&  &t^{-1}\partial_{\xi_{-}}\Big(e^{-i\eta t z } t \Psi_{t,\eta}^{\xi_{-},\xi_{+}}(tx)\Big),&
  &t^{-1}\partial_{\xi_{+}}\Big(e^{-i\eta t z } t \Psi_{t,\eta}^{\xi_{-},\xi_{+}}(tx)\Big)
\end{align*}
are uniformly bounded in $\mathcal S(\R)$ as functions of $x$. Furthermore we require that
$\widehat \Psi_{t,\eta}^{\xi_{-},\xi_{+}}$ be supported on
$(\eta-t^{-1}b,\,\eta+t^{-1}b)$ for some
$b>0$. Finally, for some constants $d,d',d''>0$ and $\epsilon>0$ it
must hold that
\begin{align}
& \label{onesummand}  \Psi_{t,\eta}^{\xi_{-},\xi_{+}}\neq 0&&\text{only if }\left\{
                                       \begin{aligned}
                                         &t(\eta-\xi_{-})\in
                                         (d-\epsilon,d+\epsilon)\\
                                         & t(\xi_{+}-\eta)>d'>0
                                       \end{aligned}
                                           \right.
  \\
  &\partial_{\xi_{+}} \Psi_{t,\eta}^{\xi_{-},\xi_{+}} =0 &&\text{if } t(\xi_{+}-\eta)>d''>d'>0. \nonumber
\end{align}
\end{remark}

\section{Localized outer-$L^p$ embeddings} \label{secloc}

We now turn to the description of the analytic tools which are relied
upon in the proof of estimate \eqref{sparsedom}. We work in the
framework of outer measure spaces \cite{DoThiele15}, see also
\cite{CuDPOu,DPOu2}. In particular, we define a distinguished
collection of subsets of the upper 3-space $\mathbb X$ which we refer
to as \emph{tents} above the time-frequency loci $(I,\xi)$ where $I$
is an interval of center $c(I) $ and length $|I|$, and $\xi \in \R$:
\begin{align*} \label{eq:def:tents}
  &\mathsf{T}(I,\xi) := \mathsf{T}^{\ell}(I,\xi ) \cup \mathsf{T}^{o}(I,\xi),
  \\
  & \mathsf{T}^{o}(I,\xi )  := \left\{ (u,t,\eta) : t\eta-t\xi \in
    \Theta^{o}, t<|I|, |u-c(I)|<|I|-t\right\} \\
  &\mathsf{T}^{\ell}(I,\xi )  := \left\{ (u,t,\eta) : t\eta-t\xi \in
   \Theta\setminus\Theta^{o}, t<|I|, |u-c(I)|<|I|-t \right\}
\end{align*}
where
$\Theta^{o}=[\beta^{-},\beta^{+}],\,\Theta=[\alpha^{-},\alpha^{+}]$
are two geometric parameter intervals such that
$0\in \Theta^{o}\subset\Theta$. The specific values of the parameters
do not matter. What is important that given the geometric parameters
of the wave packets appearing in \eqref{wpt} and \eqref{wptm} there
exists a choice of parameters of the tents such that the statements of
the subsequent discussion hold. For example it must hold that
$(-b,b)\subset\Theta^{o}$ where $b$ is the parameter that governs the
frequency support of $\psi_{t,\eta}$ and
$\Psi_{t,\eta}^{\xi_{-},\xi_{+}}$. For a complete discussion see
\cite[Sec. 2]{Ur15}.
 As
usual, we denote by $\mu$ the outer measure generated by countable
coverings by tents $\mathsf{T}(I,\xi), I\subset \R, \xi \in \R$ via
the pre-measure $ \mathsf{T}(I,\xi)\mapsto|I|$.

Let $\mathsf{s}$ be a size \cite{DoThiele15}, i.e.\  a family of
quasi-norms indexed by tents $\mathsf{T}$, defined on Borel functions
$F:\mathbb X\to \mathbb C$. The corresponding outer-$L^p$ space on
$(\mathbb X,\mu)$ is defined by the quasi-norm
\[
  \begin{split}&\|F\|_{L^{p}(\mathsf{s}) }:= \left(p\int_0^\infty
      \lambda^{p-1} \mu(\mathsf{s}(F)>\lambda) \, \d
      \lambda\right)^{\frac1p},\qquad 0<p<\infty, \\ &
    \mu(\mathsf{s}(F)>\lambda) :=\inf\Big\{ \mu(
    E): E\subset \mathbb X, \; \sup_{\mathsf{T}}
    \mathsf{s}\left(F\cic{1}_{\mathbb X\setminus E}\right)(\mathsf{T})\leq
    \lambda\Big\}
\end{split}
\]
where the supremum  on the right is taken over all tents $\mathsf{T}=\mathsf{T}(I,\xi)$.
We will work with outer $L^p$ spaces based on the sizes 
\[
\begin{split}
   \mathsf{s}^{e}(F)(\mathsf{T} )& := \; \left( \frac{1}{|I|}
      \int_{\mathsf{T}^\ell}| F(u,t,\eta) |^2 \,\d u \d t \d \eta  \right)^{\frac{1}{2}} 
    + \sup_{(u,t,\eta)\in \mathsf T} |F (u,t,\eta)|,\\
    \mathsf{s}^{m}(A)(\mathsf{T} )& :=    \left( \frac{1}{|I|}
      \int_{\mathsf{T}}| A(u,t,\eta) |^2 \,\d u \d t \d \eta  \right)^{\frac{1}{2}} 
    +   \frac{1}{|I|}
      \int_{\mathsf{T}^o}| A(u,t,\eta) |  \,\d u \d t \d \eta  
\end{split}
\]
that are related to the two embeddings \eqref{wpt} and \eqref{wptm}
respectively. The dual relation  of the sizes $\mathsf{s}^{e}, \mathsf{s}^{m}$
is given by the fact that for any  two Borel functions $F,A:\mathbb
X\to \mathbb C$ there holds 
\[ \int_{\mathsf{T}} |F(u,t,\eta)  A(u,t,\eta)|\, \d u \d t \d \eta  \leq 2\mathsf{s}^\ell(F )(\mathsf{T} )\mathsf{s}^o(A)(\mathsf{T} ).\]
The abstract \emph{outer H\"older inequality} \cite[Prop.
3.4]{DoThiele15} and Radon-Nikodym type bounds
\cite[Prop. 3.6]{DoThiele15} yield 
\begin{equation}
\label{outerholder}
\int_{\mathsf{T}}| F(u,t,\eta)  A(u,t,\eta)|\, \d u \d t \d \eta \lesssim \|F\|_{L^{\sigma }(\mathsf{s}^\ell)} \|A\|_{L^{\tau}(\mathsf{s}^o)}
\end{equation}    
whenever $1\leq\sigma,\tau \leq\infty$ are  H\"older dual exponents
i.e. $\frac{1}{\sigma}+\frac{1}{\tau}=1$.

The nature of the wave packet transforms
$f\mapsto F(f), g\mapsto A(g)$ defined by \eqref{wpt}, \eqref{wptm} is
heavily exploited in the stopping-type outer $L^p$-embedding theorems
below.  We state the embedding theorems after some necessary
definitions.
  It is convenient to use the notation
\[
\mathsf{T}(I) :=\left\{ (u,t,\eta):t<|I|,\,
  |u-c(I)|<|I|-t,\,\eta\in\mathbb R\right\}
\]
for the set of the upper 3-space associated to the usual spatial tent
over $I$.
Given an open set $E\subset \R$ we associate to it the subset of $\mathsf{T}(
E) \subset \mathbb X$ given
by
\begin{equation}
\label{tent-over-set}
\mathsf{T}(E)  = \bigcup_{I\subset E} \mathsf{T}(I)
\end{equation}
where the union is taken over all \emph{intervals} $I\subset E$.

% It is easy to see that a sufficiently large choice of $C_0$ ensures
% that $3^{8}I\subset 5Q$ for all $I \subset E_{f,p,Q}. $

%\emph{good set} is the subset of the upper 3-space
% defined as
% \begin{equation}
% \label{goodset}
% \mathsf{G}_{f,p,Q}:=\mathbb X \setminus \bigcup_{I}  \mathsf{T}(  I)
% \end{equation}
% where the union is being taken over all intervals $I$ contained in the

The first stopping embedding theorem, a reformulation of a result first obtained in \cite{DPOu2}, deals with the wave packet
transform $f\mapsto F(f)$ of \eqref{wpt}.
\begin{proposition}\label{FY}
  Let $1<p<2$, $\sigma\in (p',\infty)$, then there exists $K>0$ such
  that the following holds. For all
  $f\in L^{p}_{\mathrm{loc}}(\mathbb{R})$, all intervals $Q$, and all $c\in(0,1)$ there
  exists an open set ${U}_{f,p,Q}$ satisfying
 \begin{equation*}
  \left|{U}_{f,p,Q}\right| \leq  c |Q|,
\end{equation*}
such that 
  \begin{equation}\label{energy-embedding-bound}   \left\|F (f\cic{1}_{3Q})
      \cic{1}_{\mathsf{T}(Q)\setminus \mathsf{T}(U_{f,p,Q})}\right\|_{L^{\sigma}( \mathsf{s}^{e})}
    \leq K |Q|^{\frac1\sigma} \l f\r_{ 3Q,p}.
  \end{equation}
\end{proposition}
The embedding theorem we use to treat the variationally truncated wave
packet transform $g\mapsto A(g)$ of \eqref{wptm} stems from the main result
of \cite{Ur15}.
\begin{proposition} \label{GU} Let $\tau\in (r',\infty)$, then there
  exists $K>0$ such that the following holds. For all
  $g\in L^{1}_{\mathrm{loc}}(\R)$, all intervals $Q$ and all $c\in(0,1)$
  there exists  an open set ${V}_{g,1,Q}$ satisfying
 \begin{equation*}
  \left|{V}_{g,1,Q}\right| \leq  c |Q|,
\end{equation*}
such that 
  \begin{equation} \label{mass-embedding-bound} \left\|A (g\cic{1}_{3Q})
     \cic{1}_{\mathsf{T}(Q)\setminus \mathsf{T}(V_{g,1,Q})}\right\|_{L^{\tau}( \mathsf s^{m})} \leq
    K |Q|^{\frac1\tau} \l g\r_{3Q,1}.
\end{equation}
\end{proposition}
We stress that the constant $K$ in Proposition \ref{GU} does not depend on the parameters
$\vec a,\vec \xi,\Xi,N$ appearing in the definition \eqref{wptm} of
the map $A$.  

The above two propositions appear in \cite{Ur15} in a
somewhat different form that uses the notion of iterated outer measure
spaces introduced therein.  
We derive the statement of Propositions \ref{GU} by using the
weak boundedness on $L^{1}(\R)$ of the map \eqref{wptm} of
\cite[Theorem 1.3]{Ur15}.  In particular that result, applied to the
function  $g \cic{1}_{3Q}$  for  $\lambda=c K\l g\r_{3Q,1}$, yields a 
 collection of disjoint open intervals   $\mathcal I$ and  
 \[V_{g,1,Q}:=  \bigcup_{I \in \mathcal I} I,\qquad
 \left|V_{g,1,Q}\right| \leq { \frac{C|Q|}{K}} 
\]
so that \eqref{mass-embedding-bound} holds as required. We conclude by choosing $K\geq C/c$.
A similar procedure can be used to obtain Proposition \ref{FY} from
\cite[Theorem. 1.2]{Ur15}. 

In effect, we have shown that the
formulation of the boundedness properties of the embedding maps
\eqref{wpt} and \eqref{wptm} as expressed in Propositions \ref{FY} and
\ref{GU} are equivalent to the iterated outer measure formulation of
\cite{Ur15}. Furthermore the use of iterated outer measure $L^{p}$ norms allowed
us to bootstrap the above results to $L^{p}_{\textrm{loc}}(\R)$ generality from an a-priori type statement, as
illustrated in \cite[Section 2.1]{Ur15}.

\section{Proof of Proposition \ref{propmain}} 
\label{secpf1} 
Throughout this proof, the exponent $p\in (r',\infty)$ is fixed and
all the implicit constants are allowed to depend on $r,p$ without
explicit mention.  Since the linearization parameters play no explicit
role in the upcoming arguments we omit them from the notation, assume
them fixed and simply write $ \B(f,g)$ for the form
$ \B_{\vec \xi,\vec a } (f,g)$ defined in \eqref{defB}. Given any
interval $Q$, we introduce the localized
version
\begin{equation}
\label{therealBQ}
 \B_{Q}(f,g):=\int\displaylimits_{\mathsf{T}(Q)}
  F(f)(u,t,\eta)A(g)(u,t,\eta) \, \d u \d t \d \eta.
\end{equation}
\subsection{The principal iteration}
 The main step of the proof of Proposition \ref{propmain} is contained in the following lemma, which we will apply iteratively. 
\begin{lemma} \label{lemmaiter} There exists a positive constant $K$ such that the following holds. Let $f,g\in L^\infty(\R)$ and compactly supported, and $Q\subset\R$ be any interval.  
There exists  a countable  collection of   disjoint  open  intervals
$\mathcal{I}_{Q} $ such that 
\begin{align}
\label{lemmaiter3} &  \bigcup_{I   \in \mathcal{I}_{Q}}I\subset Q, \qquad  \sum_{I   \in \mathcal{I}_{Q}} |I| \leq 2^{-12} |Q|
\end{align}
and such that 
\begin{equation}
\label{itereq}
\B_{Q}(f\cic{1}_{3Q},g\cic{1}_{3Q}) \leq K | Q| \l f\r_{  3Q,p}  \l g\r_{ 3 Q,1} +\sum_{I \in\mathcal{I}_{Q}}   \B_{I} (f\cic{1}_{3I},g  \cic{1}_{3I}).
\end{equation}
\end{lemma}
The proof of the lemma consists of several steps, which we now begin. Notice that  there is no loss in generality with assuming that $f,g$ are supported on $3Q$: we do so for mere notational convenience.
\subsubsection{Construction of $\mathcal{I}_{ Q}$}    Referring to the notations of Section \ref{secloc},  set
\[
\begin{split}
&E_{f,Q}= U_{f,p,Q} \cup \left\{x\in \R: \mathrm{M}_p f (x) > c^{-1}\l f\r_{  3Q,p}   \right\},
\\
&E_{g,Q}= V_{g,1,Q} \cup \left\{x\in \R: \mathrm{M}_1 g (x) > c^{-1}\l g\r_{  3Q,1}   \right\},
\\ &
E_{Q}= Q\cap \left(E_{f,Q} \cup E_{g,Q}\right).
\end{split}
\]
Write the open set $E_Q$ as the union of a countable collection $I \in \mathcal{I}_{Q}$  of disjoint open  intervals.
Then \eqref{lemmaiter3} holds provided that $c$ is chosen small enough. Also, necessarily $3I\cap E_{Q}^{c}\neq \varnothing$ if $I \in \mathcal I_Q$, so that 
\begin{equation}
\label{maxfbd}
\inf_{x \in 3I} \mathrm{M}_1 f(x) \lesssim \l f \r_{3Q,p}, \qquad \inf_{x \in 3I} \mathrm{M}_1 g(x) \lesssim \l g \r_{  3Q,1}.
\end{equation}
 For further use we note that, with reference to the notations of Propositions \ref{FY} and \ref{GU},
\begin{equation}
\label{goodcover}
 \mathsf{T}(Q) \setminus\mathsf{T}(E_Q)  \subset   \mathsf{T}(Q) \setminus \left( \mathsf{T}(U_{f,p,Q})\cup \mathsf{T}(V_{g,1,Q}) \right)   
\end{equation}
  This completes the construction of $\mathcal I_Q$. 
\subsubsection{Proof of  \eqref{itereq}} We begin by using 
\eqref{tent-over-set} to partition the outer integral over $\mathsf{T}(Q)$ as 
%and the grid decomposition \eqref{gridmagic} 
\begin{equation}
\label{pfiter1}
\B_Q(f ,g )  \leq \int\displaylimits_{\mathsf{T}(Q)\setminus \mathsf{T}(E_Q)}   F(f ) A(g ) \, \d u \d t \d \eta +  \sum_{I\in \mathcal{I}_{Q}}  \B_{ I}(f ,g )  \end{equation}
Choosing $\tau\in (r',p)$, the dual exponent $\sigma=\tau'\in(p',\infty)$. By virtue of \eqref{goodcover}, we may apply the outer H\"older inequality \eqref{outerholder} and the embeddings Propositions \ref{FY} and \ref{GU} to control the first summand in 
\eqref{pfiter1} by an absolute constant times
\[
 \left\|F (f ) \cic{1}_{\mathsf{T}(Q)\setminus \mathsf{T}(U_{f,p,Q})}\right\|_{L^{\sigma}(  \mathsf s^\ell)} \left\|A (g )  \cic{1}_{\mathsf{T}(Q)\setminus \mathsf{T}(V_{g,1,Q})}\right\|_{L^{\tau}(  \mathsf s^o)} \lesssim   |Q|    \l f\r_{  3Q,p} \l g\r_{ 3 Q,1 }.\]
We turn to the second summand in \eqref{pfiter1}, which is less than or equal to
\[ \sum_{I \in \mathcal{I}_{ Q}}   \B_{I} (f\cic{1}_{3I},g \cic{1}_{3I})+ 
\sum_{\substack{(\mathsf{a},\mathsf{b}) \in \{\mathsf{in},\mathsf{out}\}^2\\ (\mathsf{a},\mathsf{b}) \neq  (\mathsf{in},\mathsf{in})}} \sum_{I \in \mathcal{I}_{ Q}}   \B_{I}(f\cic{1}_{I^{\mathsf{a}}},g \cic{1}_{I^{\mathsf{b}}}).
\]
where $I^{\mathsf{in}}=3I, I^{\mathsf{out}}=3Q\setminus 3I $.
The first term in the above display appears on the right hand side of \eqref{itereq}. We claim that
\begin{equation}
\label{tailsdom0}
\sum_{I \in \mathcal{I}_{ Q}}   \B_{I}(f\cic{1}_{I^{\mathsf{a}}},g \cic{1}_{I^{\mathsf{b}}})\lesssim |Q| \l f\r_{3Q,p}\l g\r_{ 3Q,1}, \qquad (\mathsf{a},\mathsf{b}) \neq  (\mathsf{in},\mathsf{in})
\end{equation}
thus leading to the required estimate for \eqref{itereq}.
 Assume $\mathsf{a}=\mathsf{in},\mathsf{b}=\mathsf{out}$ for the sake of definiteness, the other cases being identical. Fix $ I \in \mathcal{I}_{ Q}$. We will show that \begin{equation}
\label{tailsdom}
\B_{I}(f\cic{1}_{I^{\mathsf{in}}},g \cic{1}_{I^{\mathsf{out}}})\lesssim |I| \l f\r_{3Q,p}\l g\r_{ 3Q,1}.
\end{equation}
whence \eqref{tailsdom0} follows by summing over $I\in\mathcal I_Q$ and taking advantage of \eqref{lemmaiter3}.\subsubsection{Proof of \eqref{tailsdom}}  We introduce the Carleson box over the interval $P$
\[
\mathsf{box}(P)=\left\{(u,t,\eta)\in \mathbb X: u\in P, \textstyle \frac12 |P|\leq t< |P|\right\}
\]
Fix $I\in \mathcal I_Q$. At the root  of our argument for \eqref{tailsdom} is the fact that $\mathrm{supp} \,g \cic{1}_{I^{\mathsf{out}}}$ lies outside $3I$.  This leads to the exploitation of  the following lemma, whose proof is given at the end of the paragraph.
\begin{lemma}
\label{taillemma}
Let $P$ be any interval, $h\in L^p_{\mathrm{loc}}(\R)$, and $\tau, \sigma $ as above. There holds
\begin{align}
&\|A(h) \cic{1}_{\mathsf{box}(P)}\|_{L^{\tau}(\mathsf{s}^{m})}\lesssim
  |P|^{\frac{1}{\tau}} \left( 1+{  \frac{\mathrm{dist}(P, \mathrm{supp}
  \, h)}{|P|}} \right)^{-100} \inf_{x \in 3P} \mathrm{M}_1
  h(x), \label{tail2}
\\\label{tail1}
&
  \|F(h) \cic{1}_{\mathsf{box}(P)}\|_{L^{\sigma}(\mathsf{s}^{e})}\lesssim |P|^{\frac{1}{\sigma}}\left( 1+{  \frac{\mathrm{dist}(P, \mathrm{supp} \, h)}{|P|}} \right)^{-100} \inf_{x \in 3P} \mathrm{M}_p h(x).
\end{align}
\end{lemma}
Now let $P\in \mathcal P_k( I)$ be the collection of dyadic subintervals  of   $I$ with $|P|=2^{-k}|I|$. If $P\in\mathcal P_k (I)$ there holds  $ \mathrm{dist}(P,  I^{\mathsf{out}})\geq  |I|=2^{k}|P|  $. Moreover
\[
\sum_{P\in \mathcal P_k(I)} |P| =   |I|, \qquad \inf_{x \in 3P} \mathrm{M}_1 h(x) \lesssim 2^{k } \inf_{x \in 3I} \mathrm{M}_1 h(x) 
\]
for all locally integrable $h$. Since 
\[
\mathsf{T}(I) \subset \bigcup_{k=0}^\infty  \bigcup_{P\in \mathcal P_k( I)}\mathsf{box}(P)
\]
we obtain, using the outer H\"older inequality  \eqref{outerholder} to pass to the third line, the chain of inequalities
\[
\begin{split} &\quad 
\B_{I}(f\cic{1}_{I^{\mathsf{in}}},g \cic{1}_{I^{\mathsf{out}}}) \leq \sum_{k\geq 0} \sum_{   P\in \mathcal P_k(I)} \;\int \displaylimits_{\mathsf{box}(P)} F(f \cic{1}_{I^{\mathsf{in}}} )A(g \cic{1}_{I^{\mathsf{out}}} ) \, \d u \d t \d \eta \\ &\leq \sum_{k\geq 0} \sum_{   P\in \mathcal P_k(I)}  \|F(f \cic{1}_{I^{\mathsf{in}}}) \cic{1}_{\mathsf{box}(P)}\|_{L^{\sigma}(\mathsf{s}^{e})}
\|A(g \cic{1}_{I^{\mathsf{out}}})\cic{1}_{\mathsf{box}(P)}\|_{L^{\tau}(\mathsf{s}^{m})}
 \\ &\lesssim   \sum_{k\geq 0} \sum_{   P\in \mathcal P_k(I)} |P|
\left(\inf_{x \in 3P} \mathrm{M}_p f (x) \right)  
\left(2^{-99k}\inf_{x \in 3P} \mathrm{M}_p g (x)\right) \\ &\leq  \sum_{k\geq 0} 2^{-98k}\sum_{   P\in \mathcal P_k(I)} |P|
\left(\inf_{x \in 3I} \mathrm{M}_p f (x)\right)  
\left( \inf_{x \in 3I} \mathrm{M}_1 g   (x)\right)  \lesssim |I| \left(\inf_{x \in 3I} \mathrm{M}_p f (x)\right)  
\left( \inf_{x \in 3I} \mathrm{M}_1 g (x)\right)\end{split}
\]
which, by virtue of \eqref{maxfbd}, complies with \eqref{tailsdom}. 
 \begin{proof}[Proof of Lemma \ref{taillemma}]

   We show how estimate  \eqref{tail2} follows from Proposition \ref{GU}. Then,   \eqref{tail1} is obtained from  Proposition \ref{FY} in a
   similar manner.
   By quasi-sublinearity and
   monotonicity of the outer measure $L^{\tau}(s^{m})$ norm we have that
   \begin{equation}\label{geometric-A-decomposition}
  \begin{split}
   & \quad  \|A(h) \cic{1}_{\mathsf{box}(P)}\|_{L^{\tau}(s^{m})}\\ & \leq C
     \|A\big(h\cic{1}_{9P}\big) \cic{1}_{\mathsf{box}(P)}\|_{L^{\tau}(s^{m})}   +
     \sum_{k=3}^{\infty} C^{k}
     \|A\big(h\cic{1}_{3^{k}P\setminus 3^{k-1}P}\big) \cic{1}_{\mathsf{box}(P)}\|_{L^{\tau}(s^{m})}.
   \end{split}
    \end{equation}
   Applying the embedding bound \eqref{mass-embedding-bound} with
   $c=3^{-2}$ and $Q=3P$ provides us with $V_{h,1,3P}$ such that  $\mathsf{box}(P)\subset
   \mathsf{T}(9P)\setminus \mathsf{T}(V_{h,1,3P})$, whence
   \begin{align*}
     \|A\big(h\cic{1}_{9P}\big)\cic{1}_{\mathsf{box}(P)}\|_{L^{\tau}(s^{m})}\leq
     CK |P|^{\frac{1}{\tau}}\l h \r_{9P,1} \leq
     CK |P|^{\frac{1}{\tau}} \inf_{x\in 3P} M_{1} h(x).
   \end{align*}
   Indeed, we chose $c$ in such a way that $|V_{h,1,3P}|<3^{-1}Q$, which guarantees that $\mathsf{T}(V_{h,1,9P})$ does not intersect $\mathsf{box}(P)$.   We claim that similarly we have that for $k>2$ and for an
   arbitrarily large $N\gg 1$ there holds
   \begin{align*}
     \|A\big(h\cic{1}_{3^{k}P\setminus 3^{k-1}P}\big) \cic{1}_{\mathsf{box}(P)}\|_{L^{\tau}(s^{m})}\leq
     CK 3^{-Nk}|P|^{\frac{1}{\tau}}\l h \r_{3^{k}P,1} \leq
    C K |P|^{\frac{1}{\tau}} 3^{-Nk}\inf_{x\in 3P} M_{1} h(x).
   \end{align*}
   Let
 \[
 (u,t,\eta)\mapsto\Psi_{t,\eta}^{\xi_{-},\xi_{+}}(\cdot-u)
\]
be a choice of truncated wave packets which approximately achieves  the supremum
in $$A(h\cic{1}_{3^{k}P\setminus 3^{k-1}P})(u,t,\eta),$$ cf. \eqref{wptm}. Then
 \[
\widetilde{\Psi}_{t,\eta}^{\xi_{-},\xi_{+}}(\cdot-u):= \textstyle\left(1+ \frac{|(x-u)-c(P)|}{|P|}\right)^{2N}  \Psi_{t,\eta}^{\xi_{-},\xi_{+}}(\cdot-u)
 \]
 are adapted truncated wave packets as well since multiplying by a
 polynomial does not change the frequency support of
 $\Psi_{t,\eta}^{\xi_{-},\xi_{+}}$ and so the conditions on being
 truncated wave packets is maintained. Let $\tilde
 A(h\cic{1}_{3^{k}P\setminus 3^{k-1}P})(u,t,\eta)$ be the embedding
 obtained by using the wave packets $\tilde \Psi_{t,\eta}^{\xi_{-},\xi_{+}}(\cdot-u)$
 instead of $\Psi_{t,\eta}^{\xi_{-},\xi_{+}}(\cdot-u)$. Given that
 $(u,t,\eta)\in \mathsf{box}(P)$ we have that
 \begin{align*}
   |A(h\cic{1}_{3^{k}P\setminus 3^{k-1}P})(u,t,\eta)|\leq C3^{-2Nk} \tilde A(h\cic{1}_{3^{k}P\setminus 3^{k-1}P})(u,t,\eta).
 \end{align*}
 However the bounds \eqref{mass-embedding-bound} also hold for
 $\tilde A$ with an additional multiplicative constant that depends at
 most on $N$. Applying these bounds with $P=3^{k-1}Q$ and  $c=3^{-k}$ we have once
 again that
 \begin{align*}
   \|\tilde A(h\cic{1}_{3^{k}P\setminus 3^{k-1}P}) 
   \cic{1}_{\mathsf{box}(P)}\|_{L^{\tau}(s^{m})} \leq C K |P|^{\frac{1}{\tau}} 3^{k}
   \l h\r_{3^{k}P,1}.
 \end{align*}
As long as $N$ is chosen large enough with respect to $C>1$ appearing
in \eqref{geometric-A-decomposition}, the above display gives the required bound. The
decay factor in term of $\dist(P,\supp h)$ follows from the fact that
the the first  $k_{0}$ terms in \eqref{geometric-A-decomposition} vanish if
$\supp h \cap 3^{k_{0}}P=\varnothing$.
  \end{proof}

  \subsection{The iteration argument}With Lemma \ref{lemmaiter} in
  hand, we proceed to the proof of Proposition \ref{propmain}.  Fix
  $f,g\in L^\infty(\R)$ with compact support. By an application of
  Fatou's lemma, it suffices to prove \eqref{outer-sparsedom} with
  $\mathsf{B}_{Q_0}$ in lieu of $\mathsf B$ for an arbitrary interval
  $Q_0$ with $\supp f,\supp g \subset Q_0$. That is, it suffices to
  construct a sparse collection $\mathcal S$ such that
\begin{equation}
\label{itermain}
\B_{Q_0}(f,g)\leq C \sum_{I \in \mathcal S} |I| \l f\r_{I,p}  \l g\r_{I,1}
\end{equation}
provided that the constant $C$ does not depend on $Q_0.$ We fix such a $Q_0$. Furthermore, as
\[
\B_{Q_0}(f,g)= \sup_{\eps>0}\B_{Q_0,\eps}(f,g), \qquad  \B_{Q,\eps}(f,g):=\int_{\mathsf{T}(Q)  }  F(f)(u,t,\eta)A(g)(u,t,\eta) \cic{1}_{\{t>\eps\}} \, \d u \d t \d \eta
\] 
it suffices to prove \eqref{itermain} with $\B_{Q_0,\eps}$ replacing
$\B_{Q_0}$, with constants uniform in $\eps>0$. We also notice that
Lemma \ref{lemmaiter} holds uniformly, if one replaces all instances
of $\B_{Q}$ in \eqref{therealBQ} by $\B_{Q,\eps}$. From here onwards we
fix $\eps>0$ and drop it from the notation.

  We now perform the following iterative procedure. Set $\mathcal S_0=\{ Q_0\}$. Suppose that the collection of open intervals $Q\in \mathcal S_n$ has been constructed,
and define inductively
\[
\mathcal S_{n+1}= \bigcup_{Q\in \mathcal S_n} \mathcal{I}_{Q}
\]
where $\mathcal{I}_{Q}$ is obtained as in the Lemma \ref{lemmaiter}.   It can be seen inductively that 
\[
Q \in \mathcal S_n \implies |Q| \leq 2^{-12n}|Q_0|.\]
We iterate this procedure as long as  $n\leq N$, where $N$ is taken so that $2^{-12N}|Q_0|<\eps$ holds. At that point we stop the iteration and set
\[
\mathcal S^\star= \bigcup_{n=0}^{N} \mathcal S_{n}.
\]
Making use of estimate \eqref{itereq} along the iteration of Lemma \ref{lemmaiter}  we readily obtain
\[
\B_{Q_0}(f,g)    \lesssim   \sum_{n= 0}^{N-1} \sum_{Q \in \mathcal S_n} |Q|    \l f \r_{  3 Q,p} \l g\r_{3Q,1}  + \sum_{Q\in \mathcal S_N}\sum_{I \in\mathcal{I}_{Q}}   \B_{I} (f\cic{1}_{3I},g  \cic{1}_{3I})    =  \sum_{Q \in \mathcal S^\star} |Q|    \l f \r_{  3 Q,p} \l g\r_{3Q,1} \]
as each term $\B_I$, $I\in \mathcal S_N$ vanishes by the condition on $N$.  
Now, observing that the sets
\[
X_Q :=  Q \setminus\left(\bigcup_{I\in \mathcal S^\star: I \subsetneq Q} I\right)= Q \setminus\left( \bigcup_{I \in\mathcal{I}_{Q}} I \right)\qquad Q \in \mathcal S^\star
\]
are pairwise disjoint and, from \eqref{lemmaiter3},
$
|Q \backslash X_Q| \geq (1-2^{-12})|Q|
$
yields that  $\mathcal S^\star$ is sparse, and so is $\mathcal{S}=\{3Q:Q\in \mathcal S^\star\}$. This completes the proof  of Proposition \ref{propmain}.

%\section{Proof of Theorem \ref{thmweight}} \label{secpf2}
%

 \bibliography{DoDPUraltsev}{}
\bibliographystyle{amsplain}
 
\end{document}